\documentclass[11pt,leqno]{amsart}%
\usepackage{hyperref}
\usepackage{amsmath}
\usepackage{amsfonts}
\usepackage{amssymb}
\usepackage{graphicx}%
\providecommand{\U}[1]{\protect\rule{.1in}{.1in}}
\newtheorem{theorem}{Theorem}[section]
\newtheorem*{acknowledgement}{Acknowledgement}

\newtheorem{lemma}[theorem]{Lemma}

\newtheorem{proposition}[theorem]{Proposition}
\newtheorem{remark}[theorem]{Remark}

\newtheorem{thm}{Theorem}

\newtheorem{prp}[thm]{Proposition}
\newtheorem{prb}[thm]{Problem}

\newcommand{\tr}{\operatorname{tr}}

\newcommand{\Kp}[1]{\mathcal{K}_p\left(#1\right)}

\newcommand{\tnn}{\hat N}

\newcommand{\rr}{\mathbb{R}}
\newcommand{\hh}{\mathbb{H}}

\newcommand{\Hess}{\operatorname{Hess}}

\newcommand{\sect}{\operatorname{Sect}}

\newcommand{\nsect}{{}^N \operatorname{Sect}}

\newcommand{\vol}{\operatorname{Vol}}

\newcommand{\sss}{{\mathbb{S}}}

\newcommand{\tu}{{\tilde{u}}}
\newcommand{\tir}{{\hat{r}}}
\newcommand{\tq}{{\tilde{q}}}
\newcommand{\tv}{{\tilde{v}}}

\newcommand{\tn}{{\tilde{N}}}
\newcommand{\tm}{{\tilde{M}}}

\newcommand{\dive}{\operatorname{div}}

\newcommand{\tnmetr}[2]{\left\langle #1,#2\right\rangle_{\tn}}
\newcommand{\hsmetr}[2]{\left\langle #1,#2\right\rangle_{HS}}

\newcommand{\dist}{\operatorname{dist}}

\begin{document}

\title[On the homotopy Dirichlet problem for $p$-harmonic maps]{On the homotopy Dirichlet problem for $p$-harmonic maps}
\date{\today}
\author{Stefano Pigola}
\address{Dipartimento di Fisica e Matematica\\
Universit\`a dell'Insubria - Como\\
via Valleggio 11\\
I-22100 Como, Italy.}
\email{stefano.pigola@uninsubria.it}

\author{Giona Veronelli}
\address{
Université Paris 13, Sorbonne Paris Cité, LAGA, CNRS (UMR 7539)\\
99 avanue J-B Clément, F-93430, Villetaneuse, France.}
\email{giona.veronelli@gmail.com}


\begin{abstract}In this two papers we deal with the relative homotopy Dirichlet problem for $p$-harmonic maps from compact manifolds with boundary to manifolds of non-positive sectional curvature. Notably, we give a complete solution to the problem in case the target manifold is either compact and a new proof in case it is rotationally symmetric or two dimensional and simply connected. The proof of the compact case uses some ideas of White to define the relative d-homotopy type of Sobolev maps, and the regularity theory by Hardt and Lin. To deal with non-compact targets we introduce a periodization procedure which permits to reduce the problem to the previous one. Also, a general uniqueness result is given.\end{abstract}
\maketitle

\tableofcontents
\part{Compact targets}

\section*{Introduction}

Let $(M,g)$ and $(N,h)$ be Riemannian manifolds of dimensions $m$ and $n$ respectively. Let $u : M\to N$ be a $C^1$ map. The $p$-energy density $e_p(u) : M \to\mathbb R$ is the non-negative
function defined on $M$ as
\[
e_p (u) (x) =\frac 1 p |du|_{HS}^p(x).
\]
Here the differential $du$ is considered as a section of the $(1,1)$-tensor bundle along the map $u$, i.e. $du\in\Gamma(T^\ast M\otimes u^{-1}TN)$ is a vector valued differential $1$-form. Moreover $T^\ast M\otimes u^{-1}TN$ is endowed with its Hilbert-Schmidt scalar product.
If $\Omega\subset M$ is a compact domain, we define the $p$-energy of $u|_\Omega: \Omega \to N$ by
\[
E_p^{\Omega}(u) = \int_\Omega
e_p(u) dV_M.
\]
Let $X$ be a $C^1$ vector field along $u$, i.e. a section of the bundle $u^{-1}TN$, supported in $\Omega$. Then
\[
u_t (x) = {}^N\exp_{u(x)} t X(x).
\]
defines a variation of $u$ which preserves $u$ on $\partial\Omega$.
The map $u: M \to N$ is said to be $p$-harmonic if,
for each compact domain $\Omega\subset M$, it is a stationary point of the $p$-energy functional, that is
\begin{equation*}
\left.\frac d{dt}\right|_{t=0} E_p^\Omega (u_t) = \int_M \hsmetr{|du|^{p-2}du}{dX}dV_M=0.
\end{equation*}
The latter equality corresponds to the weak formulation of the $p$-laplacian equation
\begin{equation}\label{taupp}
\Delta_p u = \dive (|du|^{p-2}du)=0.
\end{equation}
Here $-\dive=\delta$ is the formal adjoint of the exterior differential $d$, with respect to the standard $L^2$ inner product on vector-valued differential $1$-forms on $M$.
In local coordinates, \eqref{taupp} takes the expression
\begin{align*}
(\Delta_p u)^A&=g^{ij}\left(\frac{\partial}{\partial x^j}\left(|du|^{p-2}\frac{\partial u^A}{\partial x^i}\right) - {}^M\Gamma_{ij}^k \frac{\partial u^A}{\partial x^k}|du|^{p-2} \right.\\
&\left.+ {}^N\Gamma_{BC}^A\frac{\partial u^B}{\partial x^i}\frac{\partial u^C}{\partial x^j}|du|^{p-2} \right)=0,
\end{align*}
which, in turn, can be written in the compact form
\begin{align*}
(\Delta_p u)^A=\dive\left(|du|^{p-2}\nabla u^A\right)+|du|^{p-2}\Gamma^A(du,du)=0.
\end{align*}
It's worthwhile to observe that, in case $N\hookrightarrow \mathbb{R}^q$ is isometrically embedded in some Euclidean space $\rr^q$ with second fundamental form $\mathcal{A}$, the above definition of $p$-harmonicity for a  $C^1$ map $u:M\to \rr^q$ is equivalent to the standard notion of weak $p$-harmonicity, that is
\begin{align}\label{weakpharm}
\int\left\vert Du\right\vert ^{p-2}\left\{  Du\cdot D\varphi+\mathcal{A}\left(
Du,Du\right)  \cdot\varphi\left(  x\right)  \right\} =0 ,\quad\forall\varphi\in C_c^\infty(\Omega,\rr^q)
\end{align}
where we have set
\[
Du\cdot D\varphi=g^{ij}\delta_{AB}\frac{\partial u^{A}}{\partial x^{i}}%
\frac{\partial\varphi^{B}}{\partial x^{j}}%
\]%
and
\[
\mathcal{A}\left(  Du,Du\right)\cdot \varphi  =g^{ij}\delta_{CD}\mathcal{A}_{AB}^{C}%
\frac{\partial u^{A}}{\partial x^{i}}\frac{\partial u^{B}}{\partial x^{j}%
}\varphi^{D}.
\]
To see this,  it's enough to take
\[
X_{x}=\left.  D\left(  \Pi_{N}\right)  \right\vert _{u\left(  x\right)  }%
\cdot\varphi\left(  x\right)  \in T_{u\left(  x\right)  }N\subset
\mathbb{R}^{q},
\]
where $\Pi_N$ is the nearest point projection from a tubular neighborhood of $N$ in $\rr^q$.\\
The theory of $p$-harmonic maps between Riemannian manifolds and $p$-energy minimizers has undergone a great development in the last two decades. Among the works on the subject, let us recall for instance \cite{DF-Annales,DGK-Convex,Ga-CVar,HaLi-CPAM,Na-Nonlinear,XiYa-Reine}, dedicated to the regularity theory, and \cite{Wh-Acta,DF-Asymp,PRS-MathZ,Ta-PAMS,We-Indiana,Ma,PV-GeoDed} which are concerned mostly with the connections to the geometry of the manifolds. Beside these, and in the special case of compact manifolds without boundaries, it is also worth to point out the very recent and interesting developments in the parabolic theory of $p$-harmonic maps. For instance, we quote \cite{Hu-Indiana, Mi-EJDE, FR-Indiana, FR-CVar}.\\

This paper is the first step in the investigation of the unique solvability of the homotopy
Dirichlet problem for $p$-harmonic maps into a geodesically complete manifold $N$ of non-positive curvature.

\begin{prb}\label{problem}
Let $\left(  M,g\right)  $ be a compact, $m$-dimensional Riemannian manifold
with smooth boundary $\partial M\neq\emptyset$ and let $N$ be a complete, possibly compact, $n$-dimensional Riemannian manifold without boundary. Assume also that $N$ has non-positive sectional curvature or, more generally, that the universal covering of $N$ supports a strictly convex exhaustion function. For any $p\geq2$ and any given
$f\in C^{0}\left( M,N\right)$, consider the
$p$-Dirichlet problem%
\[
\left\{
\begin{array}
[c]{ll}%
\Delta_{p}u=0 & \text{on }M\\
u=f & \text{on }\partial M.
\end{array}
\right.
\]
Has this $p$-Dirichlet problem a (unique) solution $u\in C^{1\,,\alpha}\left(  \operatorname{int} (M),N\right)  \cap C^{0}\left(  M,N\right)  $ in the homotopy class of $f$ relative to $\partial M$?
\end{prb}

Actually one expects Problem \ref{problem} to have a positive answer. A first evidence in this direction is given by the classical harmonic case. When $p=2$, the Dirichlet problem for harmonic maps into non-positively curved manifolds has been solved by R. Schoen and S.T. Yau, who extended to non-compact targets a previous result due to R. Hamilton; see \cite{Ham} and Theorem 8.5 in Chapter IX of \cite{SY-LN}. Schoen and Yau's proof makes use of Hamilton's heat flow for harmonic maps. In view of the achievements in the papers quoted above, $p$-harmonic heat flow techniques look promising in obtaining a complete solution of Problem \ref{problem} even in case $p\not=2$. However, so far and to the best of our knowledge, no significant progress in this direction has been made yet.\\
In the $p$-harmonic realm, the first interesting result is due to S.W. Wei \cite{We-Indiana} who considered, for a compact target $N$ and boundary datum $f\in C^0(M,N)\cap Lip(\partial M,N)$, a weaker version of Problem \ref{problem}. 
More precisely, in Theorem 7.1 of \cite{We-Indiana}, using a procedure similar to that introduced by F. Burstall in \cite{Burstall-London} and which is based on the homomorphisms induced by the Sobolev maps on the fundamental groups, Wei obtained solutions to the Dirichlet problem in the free homotopy class of the initial datum $f$.
If we try to solve Problem \ref{problem} even for a compact target $N$ of non-positive curvature, it is easily seen that the proof proposed by Wei can not work without changes. In fact the induced homomorphism is not enough to determine completely the relative homotopy type of a given map. In this spirit, an easy counterexample can be constructed by considering the $2$-dimensional torus $N=\mathbb{T}^2=\rr^2/\mathbb{Z}^2$ and the compact manifold with boundary $M\subset \mathbb{T}^2$ given by $M=\{[(x_1,x_2)]\in\mathbb{T}^2: 0\leq x_1 \leq 1/2\}$. Then, choosing $f:M\to N$ defined by $f([(x_1,x_2)])=[(3x_1,x_2)]$, one has that the inclusion map $i:M\hookrightarrow N$ is a harmonic map with $i|_{\partial M}=f|_{\partial M}$, which induces a homomorphism $i_\sharp$ conjugated to $f_\sharp$, but $f$ and $i$
are not homotopic relative to $\partial M$.\\
Nevertheless, since in \cite{Wh-Acta} the relative $[p-1]$-homotopy type of a $W^{1,p}$ map is defined, one can minimize the $p$-energy among maps preserving such a relative $[p-1]$-homotopy type, and show that the regularity theory of \cite{HaLi-CPAM} applies also in this case. Thus, using this strategy we are able to prove the following theorem that works out the program for non-positively curved targets proposed by B. White long ago.

\begin{thm}\label{existence_1}
Let $M$ be a compact manifold with boundary $\partial M\neq\emptyset$, and $N$ be a compact manifold whose universal covering supports a strictly convex exhaustion function. Let $f\in C^0(M,N)\cap Lip(\partial M,N) $. Then, for any $p\geq 2$, there exists a $p$-harmonic map $u\in C^{1,\alpha }(\operatorname{int} (M),N)\cap C^0(M,N)$ which minimizes the $p$-energy among all the $W^{1,p}$ maps in the homotopy class of $f$ relative to $\partial M$. In particular $u$ is the unique solution of Problem \ref{problem} when $\nsect\leq0$.
\end{thm}

We point out that the uniqueness part of the theorem is a consequence of previous results by  W. Wei, \cite{We-Indiana}. However, at the end of the paper, and in view of future developments in the non-compact case, we shall include a very general uniqueness result that works for homotopic (rel $\partial M$) $p$-harmonic maps into a complete manifold of non-positive curvature. This is the second main result of the paper which represents a new contribution to the difficult comparison theory for p-harmonic maps in the presence of topology: in fact, it extends to domains with boundary the comparison result obtained in \cite{V}.

\begin{thm}\label{th_uniqueness}
Let $\left(  M,g\right)  $ be a compact, $m$-dimensional Riemannian manifold
with smooth boundary $\partial M\neq\emptyset$ and let $N$ be a complete manifold such that $\nsect\leq 0$. Let $f\in C^0(M,N)$. Then, for any $p\geq 2$, there is at most a $C^1$ solution to the homotopy $p$-Dirichlet Problem \ref{problem} with datum $f$.
\end{thm}

\section{Existence of solutions}

The proof of Theorem \ref{existence_1}, that will be presented in Subsection \ref{subsection_proofThB}, requires some preliminary results of very different nature that are collected in the next three subsections.

\subsection{Lipschitz approximation in relative homotopy class} \label{subsection_lipApprox}

According to Section 4 in \cite{Wh-Acta} a boundary datum $f \in C^0(M,N) \cap Lip(\partial M,N)$ has a $C^0(M,N) \cap W^{1,p}(M,N)$ representative in its homotopy class relative to $\partial M$. This fact will be extensively used in the proof of Theorem \ref{existence_1}. Actually, the representative can be chosen to be $Lip(M,N)$; see Proposition \ref{prop_retraction} below. This follows by combining the standard Whitney approximation result with the next Lemma which is implicitly contained in \cite{Wh-Acta}, see especially the proof of Theorem 4.1 there.

\begin{lemma}
\label{th_retraction}Let $\left(  M,g\right)  $ be a compact $m$-dimensional Riemannian manifold with boundary $\partial
M\neq\emptyset$. Then, there exists a Lipschitz
map $u:M\rightarrow M$ satisfying the following conditions:

\begin{enumerate}
\item[(a)] $u$ is homotopic to the identity map $\mathrm{id}_{M}$ relative to
$\partial M$.

\item[(b)] $u$ is a smooth retraction of a collar neighborhood $V$ of
$\partial M$ onto $\partial M$. Moreover, the neighborhood $V$ can be chosen as small as desired.

\item[(c)] $u$ is a diffeomorphism of $M\setminus V$ onto $M$.

\end{enumerate}
\end{lemma}

\begin{proof}
Fix an open collar neighborhood $W$ of $\partial M$ and let $\alpha:\partial
M\times\lbrack0,+\infty)\rightarrow W$ be a diffeomorphism satisfying
$\alpha\left(  x,0\right)  =x$. Note that, if we set $\alpha^{-1}\left(
x\right)  =\left(  \alpha_{1}\left(  x\right)  ,\alpha_{2}\left(  x\right)
\right)  $, the map
\[
r\left(  x\right)  =\alpha\left(  \alpha_{1}\left(  x\right)  ,0\right)
:W\rightarrow\partial M
\]
is a natural smooth retraction of $W$ onto $\partial M$. Now, for any $0\leq
s<t$, let%
\[%
\begin{array}
[c]{l}%
\mathcal M_{s}^{t}=\alpha\left(  \partial M\times\lbrack s,t]\right)  \\
\mathcal M_{t}=\overline{M\backslash \mathcal M_{0}^{t}}\\
\mathcal B_{t}=\alpha\left(  \partial M\times t\right)
\end{array}
\]
so that%
\[
\mathcal M_{t}\cap \mathcal M_{s}^{t}=\mathcal B_{t}.
\]
Observe that $\mathcal M_{0}=M$ can be realized as the obvious gluing%
\[
M=\mathcal M_{2}\cup_{\mathrm{id}_{\!\mathcal B_{2}}}\mathcal M_{0}^{2}%
\]
whereas%
\[
\mathcal M_{1}=\mathcal M_{2}\cup_{\mathrm{id}_{\!\mathcal B_{2}}}\mathcal M_{1}^{2}.
\]
On the other hand, $\mathcal M_{1}^{2}$ is diffeomorphic to $\mathcal M_{0}^{2}$ via%
\[
\beta\left(  x\right)  =\alpha\left(  \alpha_{1}\left(  x\right)  ,2\alpha
_{2}\left(  x\right)  -2\right)
\]
which keeps $\mathcal B_{2}$ fixed, i.e., $\beta=\mathrm{id}_{\!\mathcal B_{2}}$ on $\mathcal B_{2}$. It
follows that the homeomorphism $\gamma:\mathcal M_{1}\rightarrow M$ such that%
\begin{equation}\label{retractiongamma}
\gamma\left(  x\right)  =\left\{
\begin{array}
[c]{ll}%
\mathrm{id}_{\!\mathcal M_{2}}\left(  x\right)  , & x\in \mathcal M_{2}\\
\beta\left(  x\right)  , & x\in \mathcal M_{1}^{2}%
\end{array}
\right.
\end{equation}
smooths out along the closed submanifold $\mathcal B_{2}=\mathcal M_{2}\cap \mathcal M_{1}^{2}$ and gives rise to
a global diffeomorphism $\Gamma:\mathcal M_{1}\rightarrow M$ satisfying $\Gamma=\gamma$
\ outside a small neighborhood of $\mathcal B_{2}$; see e.g. Theorem 1.9 of \cite{Hi}.
In particular,%
\[
\Gamma\left(  x\right)  =r\left(  x\right)  \text{, on }\mathcal B_{1}.
\]
To conclude, we put%
\[
V=\alpha\left(  \partial M\times\lbrack0,1)\right)  \subset W
\]
and define $u:M\rightarrow M$ by setting%
\[
u\left(  x\right)  =\left\{
\begin{array}
[c]{ll}%
r\left(  x\right)  , & x\in \mathcal M_{0}^{1}=\bar{V}\\
\Gamma\left(  x\right)  , & x\in \mathcal M_{1}.
\end{array}
\right.
\]
\end{proof}

\begin{remark}\label{rem_bilip_loc}
\rm{
The same proof works if $M$ is a non-compact manifold with compact boundary $\partial M$. Clearly, in this case, $u$ is only $Lip_{loc}(M,N)$.
In fact, note that the assumption that $\partial M$ is compact is just used to smoothing out the homeomorphism $\gamma$ along the  submanifold $B_2$ and this is needed to obtain $u$ satisfying the further condition $(c)$ in the statement of Lemma. If we are not interested in smooth regularity, we can use directly $\gamma$, whose construction does not require any compactness assumption on $\partial M$. In this case, condition $\mathrm{(c)}$  has to be replaced by
\begin{enumerate}
\item[(c)'] $u$ is a $BiLip_{loc}$-homeomorphism of $M\setminus V$ onto $M$.
\end{enumerate}
}
\end{remark}

\begin{proposition}
\label{prop_retraction}Keeping the notation and assumptions of the previous Lemma, suppose we are given a map
$f\in C^{0}\left(  M,N\right)  \cap Lip\left(  \partial M,N\right)  $, where $(N,h)$ is
an $n$-dimensional Riemannian manifold without boundary.
Then, there exists $F\in Lip\left(  M,N\right)  $ which is homotopic to
$f$ relative to $\partial M$.
\end{proposition}

\begin{proof}
Let $u:M\rightarrow M$ be the map defined in Lemma \ref{th_retraction}
and consider $\overline{f}=f\circ u:M\rightarrow M$. Then, $\overline{f}\in
C^{0}\left(  M\right)  \cap Lip_{loc}\left(  V\right)  $ where $V$ is a collar
neighborhood of $\partial M$ which retracts to $\partial M$ via $u$. In particular
$f=\overline{f}$ on $\partial M.$ Let $W$ be a smaller collar neighborhood of
$\partial M$ such that $\overline{W}\subset V$. Then, we can apply the
standard approximation procedure by H. Whitney keeping $\overline{W}$ fixed (see e.g. \cite{Le})
and obtain a  Lipschitz map $F:M\rightarrow N$ with the desired
properties. More precisely: (i) $F$ is smooth on $M\backslash V$; (ii) $F=\overline{f}$ on $\overline{W}$ and (iii) $F$ is homotopic to $\overline{f}$, relative to $\partial M$.
\end{proof}

\begin{remark}
\rm{
Using the version of Lemma \ref{th_retraction} observed in Remark \ref{rem_bilip_loc}, we can skip the assumption that $M$ is compact and obtain that any $f\in C^{0}\left(  M,N\right)  \cap Lip_{loc}\left(  \partial M,N\right)  $ has a representative $F\in Lip_{loc}\left(  M,N\right)  $ in its homotopy class relative to $\partial M$.
}
\end{remark}

\subsection{$p$-Minimizing tangent maps}\label{subsection_pMTM}

Another key ingredient in the proof of Theorem \ref{existence_1} is the fact that a manifold does not contain any $p$-minimizing tangent sphere provided its universal covering supports a smooth strictly convex function. This is the content of Proposition \ref{prop_p-MTM} that will be vital to apply the full regularity theory by Hardt-Lin. The conclusion for $p=2$ was initially obtained in \cite{ScUh-JDG}. Nakauchi pointed out how to extend Schoen and Uhlenbeck's iteration process to $p>2$, \cite{Na-Nonlinear}. The general result for $p>2$ and target supporting a strictly convex function was also observed in \cite{WeYa-JGA}.

Due to its importance, we shall provide a detailed and complete proof.\\

First, we introduce the Sobolev spaces of maps that will be used throughout the paper. According to the Nash embedding theorem, we can assume
that there is an isometric embedding $i:N\hookrightarrow\mathbb R^q$ of $N$ into some Euclidean space. For all maps $u:M\to N$ we define $\check u:=i\circ u:M\to\rr^q$. For $p>1$, we denote by $W^{1,p}_{loc}(M,\mathbb R^q)$ (resp. $W^{1,p}(M,\mathbb R^q)$) the Sobolev space of maps $v:M\to\mathbb R^q$ whose component functions and their first weak derivatives are in $L^p_{loc}(M)$ (resp. in $L^p(M)$). Moreover we define
\begin{align*}
&W^{1,p}_{loc}(M,N):=\{v\in W^{1,p}_{loc}(M,\mathbb R^q): v(x)\in N \textrm{ for a.e. } x\in M\},\\
&W^{1,p}(M,N):=\{v\in W^{1,p}(M,\mathbb R^q): v(x)\in N \textrm{ for a.e. } x\in M\}.
\end{align*}
Finally we will say that $v\in W^{1,p}(M,N)$ has boundary trace $f$ if $\check v-\check f\in W^{1,p}_0(M,\rr^q)$, where $W^{1,p}_0(M,\rr^q)$ denotes the closure of $C^{\infty}_c(M,\rr^q)$ in the $W^{1,p}(M,\rr^q)$ norm.

Following \cite[p. 572]{HaLi-CPAM} we say that a map $\bar\psi\in W_{loc}^{1,p}(\rr^{l+1},N)$ is a $p$-minimizing tangent map ($p$-MTM) from $\rr^{l+1}$ to $N$ if $\bar\psi$ minimizes the $p$-energy on compact sets and $\bar\psi$ is homogeneous of degree $0$, that is, $\partial\bar\psi/\partial r=0$ a.e., $r$ being the radial coordinate. Clearly, here we are thinking of $\bar \psi$ as an $\rr^q$-valued map once $N$ is isometrically embedded in the Euclidean space $\rr^q$. Note that $\bar\psi\in W_{loc}^{1,p}(\rr^{l+1},N)$ is homogeneous of degree $0$ if and only if there exists $\psi\in W^{1,p}(\mathbb S^l,N)$ such that
\begin{align}\label{homo_2}
\bar\psi(x):=\psi\left(\frac x {|x|}\right), \quad\forall x\neq 0.
\end{align}
Indeed, condition (\ref{homo_2}) clearly implies that $\partial\bar\psi/\partial r=0$ a.e. On the other hand, if $\bar\psi\in W^{1,p}_{loc}(\rr^{l+1},N)$ is homogeneous of degree $0$, since by Fubini's theorem $\bar\psi(\cdot,\theta)\in W^{1,p}_{loc}(\mathbb{R}_{> 0},N)$ for a.e. $\theta\in\sss^l$, we deduce that $\bar\psi(\cdot,\theta)$ is constant a.e. Therefore, $\psi(\theta)=\bar \psi(\cdot,\theta)$ satisfies \eqref{homo_2} and, again by Fubini, it is $W^{1,p}(\sss^l,N)$.

\begin{lemma}\label{lem_p-harm}
Assume that $\bar\psi\in W_{loc}^{1,p}(\mathbb R^{l+1},N)$ satisfies (\ref{homo_2})
for some $\psi\in W^{1,p}(\mathbb S^l,N)$. Then $\bar\psi:\rr^{l+1}\to N$ is weakly $p$-harmonic (in the sense of (\ref{weakpharm})) if and only if $\psi:\mathbb{S}^{l}\to N$ is weakly $p$-harmonic.
\end{lemma}

\begin{proof}
Let $(r,\theta) \in \mathbb{R}_{> 0} \times \sss^{l}$ be local polar coordinates on $\mathbb{R}^{l+1}$. Namely, we suppose to have chosen local angular coordinates  $\{\theta^1,...,\theta^l\}$ on $\sss^l$ so that $\{r,\theta^1,...,\theta^l\}$ is a local coordinates system for $\rr^{l+1}\setminus\{0\}$.

Having fixed an isometric embedding $i:N\to\rr^q$ with second fundamental form $\mathcal A$, we let $\check{\bar\psi}=i\circ\bar\psi:\rr^{l+1}\to\rr^q$ and $\check\psi=i\circ\psi:\sss^{l}\to\rr^q$. By assumption $\bar\psi(r,\theta)  = \psi (\theta)$, where $\psi\in W^{1,p}(\sss^l,N)$.

Let $\bar\varphi \in C^\infty_c(\rr^{l+1},\rr^q).$ Then, we have
\begin{align*}
& (D\check{\bar\psi} \cdot D\bar\varphi)(r,\theta)=
r^{-2}(D\check\psi \cdot D(\bar\varphi(r,\cdot)))(\theta),
\end{align*}
and
\[
|D\check{\bar\psi}|^2(r,\theta) = r^{-2} |D\check\psi|^2(\theta).
\]
Whence, it follows that, for every $R>0$,
\begin{align}\label{eq_pharm_energy}
&E_p(\left.  \bar{\psi}\right\vert _{\mathbb{B}_{R}\left(  0\right)  }) = E_p(\psi)  \int_{0}^R r^{l-p} dr
\end{align}
and
\begin{align}\label{eq_pharm}
&\int_{\rr^{l+1}}|D\check{\bar\psi}|^{p-2}\left\{D\check{\bar\psi} \cdot D\bar\varphi + \mathcal{A}(D\check{\bar\psi},D\check{\bar\psi}) \cdot \bar\varphi\right\} dx\\
&=\int_0^\infty r^{l-p}\int_{\sss^l}|D\check\psi|^{p-2}(\theta)\left\{D\check\psi \cdot D(\bar\varphi(r,\cdot)(\theta)\vphantom{\left\langle D\check{\bar\psi},D\bar\varphi\right\rangle_{HS(\rr^{l+1},\rr^q)}(r,\theta)}\right.\nonumber\\
&\left.\vphantom{\left\langle D\check{\bar\psi},D\bar\varphi\right\rangle_{HS(\rr^{l+1},\rr^q)}(r,\theta)}+ \mathcal{A}(D\check{\psi},D\check{\psi})\cdot \bar\varphi(r,\theta)\right\}d\sigma(\theta) dr\nonumber.
\end{align}
If $p\geq l+1$, from (\ref{eq_pharm_energy}) we must conclude that $E_p(\psi)$=0 and, therefore, that $\psi$ and $\bar\psi$ are constant. In particular, they are both trivially $p$-harmonic.\\
Suppose that $p<l+1$. Since, for each $r>0$, $\bar\varphi(r,\cdot)\in C^{\infty}(\sss^l,\rr^q)$, recalling the extrinsic definition of (weak) $p$-harmonicity given in \eqref{weakpharm}, from (\ref{eq_pharm}) we deduce that if $\psi:\sss^l\to N$ is weakly $p$-harmonic, then $\bar\psi:\rr^{l+1}\to N$ is weakly $p$-harmonic.
On the other hand, assume that $\bar\varphi$ has the form $\bar\varphi(r,\theta)=\varphi(\theta)\nu(r)$, where $\varphi\in C^\infty(\sss^l,\rr^q)$ and $\nu\in C_c^\infty([0,\infty))$ is such that $\nu^{2k+1}(0)=0$ for all $k\geq0$. Then $\bar\varphi\in C^\infty_c(\rr^{l+1},\rr^q)$ and \eqref{eq_pharm} becomes
\begin{align*}
&\int_{\rr^{l+1}}|D\check{\bar\psi}|^{p-2}\left\{D\check{\bar\psi} \cdot D\bar\varphi + \mathcal{A}(D\check{\bar\psi},D\check{\bar\psi})\cdot \bar\varphi\right\} dx\\
&=\left\{\int_0^\infty r^{l-p}\nu(r)dr\right\}\left\{\int_{\sss^l}|D\check\psi|^{p-2}\left[D\check\psi \cdot D\varphi+ \mathcal{A} (D\check\psi,D\check\psi)\cdot \varphi\right]d\sigma\right\}, \nonumber
\end{align*}
proving that $\psi:\sss^l\to N$ is weakly $p$-harmonic whenever $\bar\psi:\rr^{l+1}\to N$ is weakly $p$-harmonic.
\end{proof}

\begin{proposition}\label{prop_p-MTM}
Suppose that $N$ is compact and that  $\xi\in W_{loc}^{1,p}(\mathbb R^{l},N)$ is a $p$-MTM. If $l\leq [p]$  then $\xi$ is constant  and, if $l>[p]$ and $N$ does not support any non-constant $p$-MTM from $\rr^j$ into $N$, $j=1,...,l-1$, then $\xi$ has at most an isolated singularity at the origin. In particular, $\xi|_{\sss^{l-1}}\in C^{1,\alpha}(\sss^{l-1},N)$. Moreover, if the universal cover $\tn$ of $N$ supports a strictly convex function, then every $p$-MTM from $\rr^{l}$ to $N$ is constant, for every $l\geq 1$.
\end{proposition}
\begin{proof}
Indeed, the case $l\leq [p]$ follows directly from Theorem 4.5 of \cite{HaLi-CPAM}.
For the general case we proceed by induction. Let $l>[p]$ and suppose that every $p$-MTM from $\rr^j$ to $N$ is trivial for every $j=1,\dots,l-1$. Let  $\xi:\rr^l \to N$ be a $p$-MTM. Then, by Theorem 4.5 in \cite{HaLi-CPAM} the set of singular points of $\xi$ is discrete (possibly empty) and by homogeneity it reduces to the sole origin. In particular, by Corollary 2.6 and Theorem 3.1 in \cite{HaLi-CPAM}, we deduce that $\xi|_{\sss^{l-1}}:\sss^{l-1}\to N$ is $C^{1,\alpha}$ and it is $p$-harmonic thanks to Lemma \ref{lem_p-harm}. To conclude, in case $\tn$ supports a convex function, we can apply Theorem 1.4 in \cite{WeYa-JGA} and obtain that $\xi$ is constant.
\end{proof}

\subsection{Extending relative $d$-homotopies} \label{subsection_dHomotopies}

Fix a triangulation of $\partial M$ and extend it to a triangulation of $M$. Thus $M$ is a CW-complex and $\partial M$ is a subcomplex of $M$, see \cite{Whd-Annals, Mu}. Let $M^d$ denote that $d$-skeleton of $M$.

Two continuous maps $v,f:M \to N$ are said to be $d$-homotopic relative to $M^d \cap \partial M$ (or, equivalently, they have the same $d$-homotopy type) if there exists a continuous map $H^d:[0,1]\times M^{d} \to N$
such that $H^d(0,x)=v(x)$, $H^d(1,x)=f(x)$ for all $x\in M^d$ and $H^d(\cdot,x)=f(x)=v(x)$ for all $x\in M^d\cap \partial M$.
Clearly, when $d\geq \dim M$ the relative $d$-homotopy type of maps is nothing but the usual homotopy type relative to $\partial M$.

By the homotopy extension property of the couple $(M,M^d)$ we already know that if $v$ and $f$ have the same $d$-homotopy type, then $H^d$ extends to a full homotopy $H:M \to N$ such that $H(0,x)=v(x)$. In this subsection, under the assumption that the target manifold $N$ is aspherical, we construct a special extension $H$ of $H^{d}$ satisfying the further requirements $H(1,x)=f(x)$ for every $x\in M$ and $H(\cdot,x)=f(x)=v(x)$ for every $x\in\partial M$.

Recall that $N$ is said to be aspherical if each homotopy group $\pi_k(N)$ of $N$ is trivial for $k\geq 2$.

\begin{proposition} \label{th_dHomot}
Let $v,f\in C^0(M,N)$ and assume that $N$ is aspherical. If $v$ and $f$ have the same relative $d$-homotopy type, $d \geq 1$, then they have the same relative homotopy type.
\end{proposition}

\begin{proof}
By assumption, we know that there exists a continuous map
$H^d:[0,1]\times M^{d} \to N$
such that $H^d(0,x)=v(x)$, $H^d(1,x)=f(x)$ for all $x\in M^d$ and $H^d(\cdot,x)=f(x)=v(x)$ for all $x\in M^d\cap \partial M$.

Using the aspherical structure of $N$ in a classical manner (see e.g \cite{Hatcher}), we are going to show that $H^d$ extends to a homotopy $H^{d+1}$ between $v$ and $f$ on $M^{d+1}$ relative to $M^{d+1}\cap\partial M$, i.e., $H^{d+1}:[0,1]\times M^{d+1}\to N$ is a continuous function such that $H^{d+1}(0,\cdot)=v(\cdot)$, $H^{d+1}(1,\cdot)=f(\cdot)$ and $H^{d+1}(\cdot,x)=f(x)=v(x)$ for all $x\in M^{d+1}\cap \partial M$. Clearly we can assume $1 \leq d <m$ for otherwise there is nothing to prove.

The $(d+1)$-skeleton $M^{d+1}$ is obtained as $M^{d+1}=M^{d}\cup \left(\cup_\alpha e_\alpha^{d+1}\right)$, where $e_\alpha^{d+1}$ are open $(d+1)$-cells with attaching maps $\psi_\alpha:\mathbb{S}^d\to M^d$ and corresponding characteristic maps $\bar\psi_\alpha: \mathbb{D}^{d+1} \to M^{d+1}$. Note that $[0,1]\times M^d \subseteq ([0,1]\times M)^{d+1}$ and $H^d$ extends to a continuous function $H^{d+1}:([0,1] \times M)^{d+1} \to N$ by setting $H^{d+1}(0,x)=v(x)$ and $H^{d+1}(1,x)=f(x)$ on $\{0,1\}\times \bar e_\alpha^{d+1}$.

The CW complex $[0,1]\times M^{d+1} \subseteq ([0,1]\times M)^{d+2}$ is obtained by attaching to $([0,1]\times M)^{d+1}$ the open cells $E_\alpha^{d+2} = (0,1)\times e_{\alpha}^{d+1}$ via $\Psi_\alpha:\sss^{d+1} \to ([0,1]\times M)^{d+1}$ where  $\sss^{d+1} \approx ([0,1]\times\sss^{d}) \cup (\{0,1\}\times \mathbb{D}^{d+1})$, $\Psi_\alpha = \mathrm{id}\times \psi_\alpha$ on $[0,1]\times\sss^{d}$ and $\Psi_\alpha = \{0,1\}\times \bar\psi_\alpha$ on $\{0,1\}\times \mathbb{D}^{d+1}$.
Let us show how to define $H^{d+1}$ on $([0,1]\times M)^{d+1} \cup E_\alpha^{d+2}$. Since $\partial M$ is a subcomplex of $M$, we have that either $e_\alpha^{d+1}\subset \partial M$ or $e_\alpha^{d+1}\subset \operatorname{int}(M)$. In the first case we just define, for all $(t,x) \in E_\alpha^{d+2}$, $H^{d+1}(t,x)=f(x)=v(x)$. In the second case, note that $H^{d+1}$ is already defined on $\Psi_\alpha(\sss^{d+1})$. Since $d+1\geq 2$ and $N$ is aspherical, the composition $H^{d+1}\circ \Psi_\alpha: \sss^{d+1} \to N$ is null-homotopic and, therefore, it extends to a continuous map $H^{d+1}\circ \Psi_\alpha: E_\alpha^{d+2} \to N$. This implies that $H^{d+1}$ itself extends to a continuous map from $([0,1]\times M)^{d+1} \cup E_\alpha^{d+2}$ into $N$.

Now, repeating the same procedure inductively for all $E_\alpha^{d+2}$ and for all the $k$-skeletons $M^k$, $d<k\leq m$, we complete the construction of the desired relative homotopy $H$.

\end{proof}

\subsection{The $d$-homotopy type of $W^{1,p}$ maps}
Let $d$ be the greatest integer less than or equal to $p-1$ and let $M^d$ be a $d$-dimensional skeleton of $M$. Clearly, here we mean $M^d\equiv M$ for $p-1>m$. Recall from the previous subsection that the $d$ homotopy type of a continuous map from $M$ to $N$ is the homotopy type of its restriction to $M^d$. According to the work of White \cite{Wh-Acta}, each $u\in W^{1,p}(M,N)$ with boundary trace $f$ has a $d$-homotopy type $u_\sharp[M^d(\operatorname{rel} \partial M)]$. This $d$-homotopy type is a homotopy class (relative to $\partial M$) of continuous mappings from $M^d$ into $N$ such that:
\begin{enumerate}
	\item If $\{u_i\}\subset W^{1,p}(M,N)$ have boundary trace $h$, $\|\check u_i-\check u\|_p\to0$, and $\|du_i\|_p$ is uniformly bounded, then
	\[
	(u_i)_\sharp[M^d(\operatorname{rel} \partial M)]=u_\sharp[M^d(\operatorname{rel} \partial M)]
	\]
	for sufficiently large $i$.
	\item If $u\in W^{1,p}(M,N)$ has boundary trace $f$ and is continuous at each $x\in M^d$, then
	\[
	u_\sharp[M^d(\operatorname{rel} \partial M)]=[(u|_{M^d})(\operatorname{rel} \partial M)].
	\]
	\item The set
	\[
	\{u_\sharp[M^d(\operatorname{rel} \partial M)]:u\in W^{1,p}(M,N) \textrm{ has boundary trace }f\}
	\]
	is equal to
	\begin{align*}
	&\{[(\varphi|_{M^d})(\operatorname{rel} \partial M)]:\varphi\in C^0(M^{d+1},N),\\
	&\ \varphi(x)=f(x)\textrm{ for }x\in M^{d}\cap \partial M\}.
	\end{align*}
\end{enumerate}

The purpose of this subsection is to point out the following property of the $d$-homotopy type whose application will be basic in the proof of Theorem \ref{existence_1} to apply the regularity theory of \cite{HaLi-CPAM}.

\begin{proposition}\label{pr_perturbing}
Let $M$ be a compact $m$-dimensional manifold with (possibly empty) boundary $\partial M$ and let $N$ be a compact manifold. Let $f\in Lip(\partial M,N)$ and let $v\in W^{1,p}(M,N)$ be a map with boundary trace $f$. For every $x\in M$ there exists an open set $x\in\Omega_x\subset M$ (independent of $v$) with smooth boundary $\partial\Omega_x$, which satisfies the following property: for any other map $w\in W^{1,p}(M,N)$ such that $w|_{\partial M}=v|_{\partial M}$ in the trace sense and $w\equiv v$ on $M\setminus\Omega_x$ it holds
\begin{equation}\label{perturbing_homot}
w_\sharp[M^d(\operatorname{rel} \partial M)]=v_\sharp[M^d(\operatorname{rel} \partial M)].
\end{equation}
\end{proposition}

\begin{proof}
We consider three cases.\\
First suppose that $x\in\operatorname{int}(M)$ and that $x\not\in M^d$. Since $M^d$ is closed in $M$ we can choose the open set $\Omega_x$ such that $x\in\Omega_x\subset\subset \operatorname {int} (M)\setminus M^d$ and $\partial \Omega_x$ is smooth. By the construction of the $d$-homotopy type given in Section 3 of \cite{Wh-Acta}, and up to choosing $\delta>0$ small enough in (the version for manifolds with boundary of) Proposition 3.2 therein, it is clear that in this case perturbing a $W^{1,p}$ map in $\Omega_x$ does not affect the $d$-homotopy type of the map.\\
Now, let $x\in M^d\cap\operatorname {int}(M)$. Consider, for $\epsilon>0$ small enough, a normal closed geodesic ball $\bar B_\epsilon(x)$ centered at $x$. Choose a triangulation $T_x$ of $\bar B_\epsilon(x)$ such that $x$ is not contained in the $d$-skeleton $T^d_x$ of $T_x$ (to this purpose, one can for instance take such a construction on the Euclidean unit closed ball, and make use of the diffeomorphism with $\bar B_\epsilon(x)$ given by the normal coordinates). Note that $T_x$ induces a triangulation of $\partial \bar B_\epsilon(x)$. Choosing a triangulation of $\partial M$ gives, together with $T_x|_{\partial \bar B_\epsilon(x)}$, a triangulation of $\partial M\cup\partial \bar B_\epsilon(x)=\partial (M\setminus \bar B_\epsilon(x))$. A classical result ensures us that this triangulation of the boundary can be extended to a triangulation of all of $M\setminus \bar B_\epsilon(x)$ \cite{Whd-Annals,Mu}. This latter, together with $T_x$, forms a new triangulation of $M$ whose $d$-skeleton $(M^d)'$ does not contain $x$. According to the previous case there exists an open set with smooth boundary $\Omega_x$ such that, given maps $v$ and $w$ as in the statement, we have
\[
w_\sharp[(M^d)'(\operatorname{rel} \partial M)]=v_\sharp[(M^d)'(\operatorname{rel} \partial M)].
\]
To conclude this case, we recall that, thanks to Proposition 3.5 of \cite{Wh-Acta}, the $d$-homotopy type of a $W^{1,p}$ map does not depend on the choice of the triangulation.\\
Finally, suppose that $x\in\partial M$. Using the notation of Lemma \ref{th_retraction}, let $y\in\mathcal \partial V$ be a point satisfying $u(y)=x$. Since $y\in\operatorname{int} (M)$, according to previous paragraphs there exists an open set $\Omega'_y$, $y\in\Omega'_y\subset\subset\operatorname{int} (M)$, such that perturbing a $W^{1,p}$ map inside $\Omega'_y$ does not change the $d$-homotopy type of the map.
Let $\Omega_x=u(\Omega'_y\cap (M\setminus V))$. Since $u|_{M\setminus V}$ is a differomorphism onto $M$, then $\Omega_x$ is an open set of $M$ containing $x$. Suppose that $v,w$ are two $W^{1,p}(M,N)$ maps as in the statement. Then, also $v\circ u$ and $w\circ u$ are maps in $W^{1,p}(M,N)$. Furthermore, by construction of $u$, it holds $v\circ u|_{V}=w\circ u|_{V}=f\circ u|_{V}$, where $f$ is the trace value of $v$ and $w$ at the boundary.
This in turn implies $v\circ u|_{\partial M}=w\circ u|_{\partial M}$ in the trace sense, and $v\circ u(z)= w\circ u(z)$ for each $z\in M\setminus \Omega_y'$, since either $u(z)\in\partial M$ or $u(z)\in M\setminus \Omega_x$. Hence,
\[
(w\circ u)_\sharp[{M^d}(\operatorname{rel} \partial M)]=(v\circ u)_\sharp[{M^d}(\operatorname{rel} \partial M)].
\]
On the other hand, by the construction of the relative $d$-homotopy type of $W^{1,p}$ maps given in \cite{Wh-Acta} it is clear that
\begin{align*}
&(v\circ u)_\sharp[{M^d}(\operatorname{rel} \partial M)]=v_\sharp[{M^d}(\operatorname{rel} \partial M)],\\
&(w\circ u)_\sharp[{M^d}(\operatorname{rel} \partial M)]=w_\sharp[{M^d}(\operatorname{rel} \partial M)].
\end{align*}
This latter, in turn, implies \eqref{perturbing_homot} as aimed. To conclude, observe that since $\partial M$ is smooth, up to possibly restrict the set $\Omega_x$, we can require that $\Omega_x$ has smooth boundary.
\end{proof}

\subsection{Proof of Theorem \ref{existence_1}} \label{subsection_proofThB}
For the sake of clarity, we will divide the proof in four steps.\\

\textbf{Step 1. Existence of a minimizer in the $d$-homotopy class of $f$.}
Define $\mathcal{H}_f^d$ as the space of maps $u\in W^{1,p}(M,N)$ such that $u|_{\partial M}=f|_{\partial M}$ in the trace sense and $f$ and $u$ have the same relative $d$-homotopy type, i.e.
\begin{align*}
\mathcal H^d_f := &\{u\in W^{1,p}(M,N) : \check u- \check f \in W^{1,p}_0(M,\rr^q)\textrm{ and}\\
&u_\sharp[M^d(\operatorname{rel} \partial M)]=f_\sharp[M^d(\operatorname{rel} \partial M)]\}.
\end{align*}
According to Proposition \ref{prop_retraction}, there is no loss of generality if we assume that $f \in Lip(M,N)$. In particular  $f\in\mathcal{H}^d_f $ and, therefore,
\[
\mathcal I^d_f:=\inf_{u\in\mathcal H^d_f}E_p(u)<+\infty.
\]
Let $\{v_j\}_{j=1}^{\infty}\subset\mathcal H^d_f$ be a sequence minimizing the $p$-energy in $\mathcal H^d_f$, i.e. $E_p(v_j)\to \mathcal I^d_f$ as $j\to\infty$. For the ease of notation, throughout all the proof we will keep the same set of indexes each time we will extract a subsequence from a given sequence.

Since $N$ is compact, then $\{\check v_j\}_{j=1}^{\infty}$ is bounded in $W^{1,p}(M,\mathbb R^q)$ and, up to choosing a subsequence, $\check v_j$ converges to some $\check v\in W^{1,p}(M,\rr^q)$ weakly in $W^{1,p}$. Since  $M$ is compact, $\{\check v_j\}_{j=1}^{\infty}$ is bounded in $W^{1,p'}(M,\rr^q)$ for every $p'\leq p$ which satisfies also $p'<m$. By the Kondrachov theorem, \cite{A} p.55, $\check v_j$ converges strongly in $L^s(M,\rr^q)$ for any $1<s<(mp')/(m-p')$, notably for $s=p$, and hence pointwise almost everywhere. Since $N$ is properly embedded, this implies $\check v(x)\in N$ for a.e. $x\in M$, so that we can define $v\in W^{1,p}(M,N)$ by $v=\check v$. Since $\check v_j- \check f \in W^{1,p}_0(M,\rr^q)$ for all $j$, the weak limit $\check v- \check f \in W^{1,p}_0(M,\rr^q)$.\\
By the lower semicontinuity of $E_p$ we have
\begin{equation}\label{lsc}
E_p(v)\leq \liminf_{j\to\infty}E_p(v_j)= \mathcal I^d_f.
\end{equation}
Since $\{\check v_j\}_{j=1}^{\infty}$ is bounded in $W^{1,p}(M,\mathbb R^q)$ and $\|\check v_j-\check v\|_p\to 0$ as $j\to\infty$, by the property (1) of the $d$-homotopy type of maps we deduce that
\[
v_\sharp[M^d(\operatorname{rel} \partial M)]=(v_j)_\sharp[M^d(\operatorname{rel} \partial M)]=f_\sharp[M^d(\operatorname{rel} \partial M)]
\]
for $j$ large enough, which implies that $v\in \mathcal H^d_f$.
It then follows from \eqref{lsc} that
\begin{align*}
\mathcal I^d_{f} \leq  E_p(v) \leq \mathcal I^d_f,
\end{align*}
so that $E_p(v)=\mathcal I^d_f$, i.e. $v$ minimizes the energy in $\mathcal H_{f}^d$.\\

\textbf{Step 2. Regularity of the minimizer.} We show that the regularity theory of \cite{HaLi-CPAM} applies to $v$, as already remarked on page 3 of \cite{Wh-Acta}. Clearly, the only interesting case is  $d:=[p]-1\leq m-1$.

First of all, we note that, for every $x\in M$, there exists an open set with smooth boundary $\Omega_x\ni x$ such that $v|_{\Omega_x}$ is a minimizer for the $p$-energy among all the maps $w\in W^{1,p}(\Omega_x,N)$ which have the same trace boundary of $v$ on $\partial\Omega_x$, that is $v|_{\partial\Omega_x}=w|_{\partial\Omega_x}$ in the trace sense. To see this, let $\Omega_x$ be the open set given by Proposition \ref{pr_perturbing}.
We can extend $w$ to $\bar w\in W^{1,p}(M,N)$ by setting $\bar w=v$ on $M\setminus \Omega_x$. An application of Proposition \ref{pr_perturbing} gives that $\bar w\in \mathcal H^d_f$. Then, $E_p(v)\leq E_p(\bar w)$. To conclude, we note that
\[
E_p(v|_{B_\epsilon})+E_p(v|_{M\setminus B_\epsilon})= E_p(v)\leq E_p(\bar w) = E_p(w) + E_p(v|_{M\setminus B_\epsilon}).
\]
This minimizing property enables us to apply the partial interior regularity and deduce that the singular set $\mathcal{S}(v)$ of $v$ is empty if $p>m$ and
it is a relatively closed subset of zero $(m-p)$-Hausdorff dimension if $p\leq m$. Moreover, $v$ is $C^{1,\alpha}$ on $\operatorname{int} (M) \setminus \mathcal{S}(v)$; see Corollary 2.6 and Theorem 3.1 in \cite{HaLi-CPAM}.

The full interior regularity is now obtained from Theorem 4.5 of  \cite{HaLi-CPAM} because, according to Proposition \ref{prop_p-MTM} above, every $p$-minimizing tangent map $\xi:\mathbb{R}^{l+1} \to N$ is constant, for every $l\geq 1$.

Finally, we observe that the boundary regularity theory developed in Section 5 of \cite{HaLi-CPAM} works for a Lipschitz boundary datum $f$. Therefore we can conclude that the minimizer $v$ is $C^{0,\alpha}$ on $M$.\\

\textbf{Step 3. On the relative homotopy class of the minimizer.} It remains to prove that the minimizer $v$ is homotopic to the datum $f$ relative to $\partial M$. To this end, recall that $M$ is realized as a polyhedral complex, hence a CW complex, in such a way that $\partial M$ is a subcomplex. By construction, we know that $v$ has the same $d(\geq1)$-homotopy type of $f$ relative to $M^d \cap \partial M$. Note also that $N$ is aspherical. Indeed, since its universal covering $\tn$ supports a strictly convex exhaustion function, by standard Morse theory $\tn$ is diffeomorphic to $\rr^n$. The desired conclusion now follows from a direct application of Proposition \ref{th_dHomot}.\\

\textbf{Step 4. Non-positively curved targets.} Suppose now that the compact manifold $N$ has non-positive sectional curvature so that, in particular, its universal covering $\tn$ is a Cartan-Hadamard manifold. By the Hessian comparison theorem, the square of the distance function on  $\tn$ is a strictly convex exhaustion function. Therefore, by the preceding steps, the homotopy $p$-Dirichlet problem has a solution $v \in C^{1,\alpha}(\operatorname{int} (M)) \cap C^0(M)$. Applying Theorem 8.5 (1) of \cite{We-Indiana} we conclude that such a solution is unique.

This completes the proof of the Theorem.

\section{A general uniqueness result}
In Theorem \ref{existence_1}, the uniqueness property enjoyed by solutions of the $p$-Dirichlet problem is obtained from a result by W. Wei. In this Section we extend Wei's result to solutions of the homotopic $p$-Dirichlet problem in case the target manifold is non-compact. The construction via the quotient manifold $\tnn$ proposed in the proof below comes back to Schoen and Yau \cite{SY-Topo}, which studied the moduli space of harmonic maps when $M$ is a complete non-compact manifold with finite volume. Subsequently, in \cite{PRS-MathZ} it was observed that it is enough for $M$ to be parabolic, while a generalization of Schoen and Yau's uniqueness results to $p$-harmonic maps has been obtained in \cite{V}. In particular, in \cite{V} there were introduced the convexity result stated below as Lemma \ref{lem_hess_p} and the ``mixed'' vector field $X$ used here, which in turn inspires to \cite{PRS-MathZ} and \cite{HPV-Pota}.

\begin{proof}[Proof (of Theorem \ref{th_uniqueness})]
Suppose $u$ and $v$ are two $C^{1}(M,N)$ solutions to Problem \ref{problem}. Let $P_M:\tm\to M$ and $P_N:\tn\to N$ be the universal Riemannian covers of $M$ and $N$, respectively. Note that $\tm$ is a simply connected manifold with non-empty boundary $\partial\tm$ (which is in general neither compact nor simply connected) such that $P_M(\partial\tm)\equiv \partial M$.\\
The fundamental groups $\pi_1(M,\ast)$ and $\pi_1(N,\ast)$ act as groups of isometries on $\tm$ and $\tn$ respectively, so that $M = \tm/\pi_1(M, \ast)$ and
$N = \tn/\pi_1(N, \ast)$. Let $\dist_{\tn}: \tn \times\tn\to\mathbb R$ be the distance function on $\tn$. Since ${}^{\tn}\sect\leq0$, we know that $\dist_{\tn}$ is smooth on $(\tn\times\tn)\setminus\tilde D$, where $\tilde D$ is the diagonal set $\{(\tilde x,\tilde x):\tilde x\in\tn\}$, and $\dist_{\tn}^2$ is smooth on $\tn\times\tn$. Now $\pi_1(N,\ast)$ acts on $\tn\times\tn$ as a group of isometries by
\[
\beta(\tilde x,\tilde y) = (\beta(\tilde x),\beta(\tilde y))\qquad\textrm{for }\beta\in\pi_1(N,\ast).
\]
Thus $\dist_{\tn}^2$ induces a smooth function
\[
\tir^2 : \tnn\to\mathbb R,
\]
where we have defined
\[
\tnn:=(\tn\times\tn) / \pi_1(N,\ast).
\]
Let $U: M\times [0, 1]\to N$ be a continuous relative homotopy between $u$ and $v$.
Since $\tilde{M}$, hence $\tilde{M}\times [0,1]$, is simply connected, $U$ lifts to a homotopy $\tilde{U}$ between $\tilde{u}\left(\cdot\right):=\tilde{U}\left(\cdot,0\right)$ and $\tilde{v}\left(\cdot\right):=\tilde{U}\left(\cdot,1\right)$ relative to $\partial \tilde{M}$. Clearly, $P_N(\tilde{u})=u(P_M)$ and $P_N(\tilde{v})=v(P_M)$.
Since Riemannian coverings are local isometries, $\tu$ and $\tv$ are $p$-harmonic maps and
\[
|d\tu|(\tq)=|du|(P_M(\tq)),\qquad |d\tv|(\tq)=|dv|(P_M(\tq)).
\]
Now, $\pi_1(M,\ast)$ acts as a group of isometries on $\tm$ and we have
\begin{equation}\label{gamma}
\tu(\gamma(\tq))=u_\sharp(\gamma)\tu(\tq),\qquad \tv(\gamma(\tq))=v_\sharp(\gamma)\tv(\tq),\qquad\forall\tq\in\tm,\gamma\in\pi_1(M,\ast),
\end{equation}
where $u_\sharp,v_\sharp:\pi_1(M,\ast)\to\pi_1(N,\ast)$ are the induced homomorphism and $u_\sharp\equiv v_\sharp$ since $u$ is homotopic to $v$.\\
Thus, the map $\tilde j: \tm\to\tn\times\tn$ defined by $\tilde j(\tilde x) := (\tilde u(\tilde x), \tilde v(\tilde x))$ induces via (\ref{gamma}) a map
\[
j: M \to \tnn.
\]
Furthermore, we can construct a vector valued $1$-form $J\in T^\ast M\otimes j^{-1}T\tnn$ along $j$ by projecting via \eqref{gamma} the vector valued $1$-form $\tilde J$ along $\tilde j$ defined as
\[
\tilde J := (\Kp\tu,\Kp\tv)\quad \in T^\ast \tm\otimes \tilde j^{-1}T\left(\tn\times\tn\right).
\]
Here and on, the symbol $\Kp\tu$ stands for
\[
\Kp\tu:=|d\tu|^{p-2}d\tu.
\]
Consider the vector field on $M$ given by
\begin{equation*}
X|_{q}:=\left[d\tir^2|_{j(q)}\circ J|_{q}\right]^{\sharp}.
\end{equation*}
Note that
\begin{equation}\label{field_X'}
X|_{q}:=dP_M|_{\tq} \circ \left.\tilde X\right|_{\tq},
\end{equation}
where
\[
\left.\tilde X\right|_{\tq}:=\left[\left.d\left(\dist_{\tn}^2\right)\right|_{\tilde j(\tq)}\circ \left.\tilde J\right|_{\tq}\right]^{\sharp}.
\]
We claim that \eqref{field_X'} is well defined. To this end, let $S_{\tq}\in T_{\tq}\tm$ be an arbitrary vector and let $\tq'\in P_M^{-1}(q)\subset T\tm$. If $\tq'\neq\tq$, there exists $\gamma\in\pi_1(M,\ast)$ such that $q'=\gamma q$. Then,
\[
\tilde J|_{\gamma\tq}(d\gamma\left( S_{\tq})\right)=\left(d\left[u_{\sharp}(\gamma)\right]\left(\Kp{\tu}(S_{\tq})\right),d\left[v_{\sharp}(\gamma)\right]\left(\Kp{\tv}(S_{\tq})\right)\right).
\]
Since $u$ is homotopic to $v$, $u_{\sharp}=v_{\sharp}$. Moreover $\dist_{\tn}$ is equivariant with respect to the action of $\pi_1(N)$ on $\tn\times\tn$, i.e.
\[
\dist_{\tn}(\beta \tilde x_1, \beta \tilde x_2)=\dist_{\tn}( \tilde x_1, \tilde x_2),\qquad\forall\beta\in\pi_1(N),\ x_1,x_2\in\tn.
\]
Then
\[
dP_M|_{\tq}\circ \left[\left.d\left(\dist_{\tn}^2\right)\right|_{\tilde j(\tq)}\circ \left.\tilde J\right|_{\tq}\right]^{\sharp}
\]
does not depend on the choice of $\tq\in P_M^{-1}(q)$.\\
Now, we recall the following ``convexity'' result of \cite{V}.
\begin{lemma}\label{lem_hess_p}
For all $q\in M$ and for any choice of $\tq\in P_M^{-1}(q)$ we have
\begin{equation}\label{eq_coro}
\tr_\tm {}^{\tn\times \tn}\Hess \dist_{\tn}^2|_{\tilde j(\tq)}\left(d\tilde j,\tilde J\right)\geq 0
\end{equation}
Moreover, having fixed an orthonormal frame $\tilde E_i$ in $T_\tq\tm$, with $i=1,\dots,m$, the equality holds in \eqref{eq_coro} if and only if there are parallel vector fields $Z_i$, defined along the unique geodesic $\gamma_{\tq}$ in $\tn$ joining $\tilde u(\tq)$ and $\tilde v(\tq)$, such that $Z_i(\tu(\tq))=d\tu|_{\tq}(\tilde E_i)$, $Z(\tv(\tq))=d\tv|_{\tq}(\tilde E_i)$ and $\tnmetr{{}^\tn R(Z_i,\dot{\gamma}_{\tq})\dot{\gamma}_{\tq}}{Z_i}\equiv0$ along $\tilde\gamma_{\tq}$. Moreover, $d(\operatorname{\dist}_{\tn}(\tilde j))=0$.\\
In particular, if $\nsect<0$, $Z_i$ is proportional to $\dot{\gamma}_{\tq}$ for each $i=1,\dots,m$.
\end{lemma}
By the homotopy assumption, for each $q\in\partial M$ and any $\tilde{q} \in P_M^{-1}(q)$ we have $\tilde{u}(\tilde{q})=\tilde{v}(\tilde{q})$, i.e., $\tilde{j}(\tilde{q})\in \tilde D$. In particular, this implies that $\hat{r}^2(j)|_{\partial M}=0$, and, since $d\hat{r}^2=2\hat{r} d\hat{r}$,
\[
X|_{\partial M}=0.
\]
Then, applying the divergence theorem,
\begin{equation}\label{stokes}
\int_M \dive X dV_M= 0.
\end{equation}
On the other hand, by the $p$-harmonicity of $u$ and $v$ and by the isometry property of the coverings projections,
\begin{align}\label{div_hes}
{}^M\dive X|_q &= \tr_M{}^{\tnn}\Hess \tir^2|_{j(q)}\left(dj,J\right) + d\tir^2|_{j(q)}(\dive J|_q)\\&= \nonumber
\tr_M{}^{\tnn}\Hess \tir^2|_{j(q)}\left(dj,J\right)\\
&= \nonumber
\tr_\tm{}^{\tn\times\tn}\Hess \dist_\tn^2|_{\tilde j(\tq)}\left(d\tilde j|_\tq,\tilde J|_\tq\right)
\end{align}
for each $q\in M$ and any $\tq\in P_M^{-1}(q)$. By Lemma \ref{lem_hess_p} we thus get $\dive X\geq 0$ and \eqref{stokes} implies $\dive X\equiv 0$. Thus \eqref{div_hes} holds with the equality sign and the equality conditions in Lemma \ref{lem_hess_p} give $d\left(\dist_\tn\right)(d\tu,d\tv)\equiv 0$. Since $\dist_\tn(\tu,\tv)|_{\partial \tm}=0$ we get $\tu\equiv\tv$ and, projecting on $M$, $u\equiv v$.\\
To conclude the proof, let us remark that in general relations \eqref{stokes} and \eqref{div_hes} has to be considered in the weak sense. Lemma 7 in \cite{V} proves the weak validity of \eqref{div_hes}, i.e.
\begin{align}\label{weak_div}
-\int_M \left[d\tir^2|_{j}\circ J\right]({}^M\nabla\eta)=\int_M \eta\;\tr_M{}^{\tnn}\Hess \tir^2|_{j(q)}\left(dj,J\right)
\end{align}
for all $\eta\in C^{\infty}_0(M)$. Moreover we can choose a $1$-parameter family of smooth cut-off functions $\{\eta_\epsilon\}$ compactly supported  in $\operatorname{int} (M)$ such that $\sup_M|\nabla\eta_\epsilon|=O(\epsilon^{-1})$ as $\epsilon\to 0$ and $\eta_\epsilon(q)=1$ for all $q\in M$ satisfying $\dist_M(q,\partial M)>\epsilon$. Since $X|_{\partial M}\equiv0$, $X$ is continuous and
\[
\vol_M(\{q\in M:\dist_M(q,\partial M)\leq\epsilon\})=O(\epsilon)
\]
as $\epsilon\to 0$, applying \eqref{weak_div} with $\eta=\eta_\epsilon$ and letting $\epsilon\to 0$, we can conclude that the LHS of \eqref{weak_div} tends to $0$. In some sense this gives a weak version of \eqref{stokes}. On the other hand by Lemma \ref{lem_hess_p}, we can apply monotone convergence to the RHS of \eqref{weak_div} to get
\[
\int_M \tr_M{}^{\tnn}\Hess \tir^2|_{j(q)}\left(dj,J\right) =0.
\]
\end{proof}

\part{Cartan-Hadamard targets with special structure}\label{part1}

\section*{Introduction}

Let $(M,g)$ and $(N,h)$ be Riemannian manifolds of dimensions $m$ and $n$ respectively and suppose that $M$ is compact with smooth nonempty boundary. A $C^1$ map $u : \textrm{int} M\to N$ is said to be $p$-harmonic if it satisfies the $p$-Laplace equation
\begin{equation}\label{taupp}
\Delta_p u = \dive (|du|^{p-2}du)=0.
\end{equation}
Here $du\in\Gamma(T^\ast M\otimes u^{-1}TN)$ is a vector valued differential $1$-form and $T^\ast M\otimes u^{-1}TN$ is endowed with its Hilbert-Schmidt scalar product. Moreover $-\dive=\delta$ is the formal adjoint of the exterior differential $d$, with respect to the standard $L^2$ inner product on vector-valued differential $1$-forms on $M$. Equation \eqref{taupp} is the Euler-Lagrange equation of the $p$-energy functional
\[
E_p(u) = \frac 1 p \int_\Omega
|du|_{HS}^p(x)dV_M.
\]
The topic of this paper is
 the Dirichlet problem for $p$-harmonic maps into nonpositively curved target $N$. Namely, given a sufficiently regular boundary datum $f:\partial M\rightarrow N$ the
corresponding Dirichlet problem consists in finding a map $u:M\rightarrow N$ which extends $f$ to a harmonic map on $\mathrm{int}M$.

\medskip

In case the target manifold $N$ is closed (i.e. compact without boundary), in Part 1 we gave a complete solution to the homotopic $p$-Dirichlet, i.e. the solution is found in a prescribed homotopy class. The proof therein is purely variational. Exploiting powerful techniques due to B. White, \cite{Wh-Acta}, one can define the weak relative $d$-homotopy type of $W^{1,p}$ maps, hence minimize
the $p$-energy in the $d$-homotopy class of the initial datum, and finally show how to apply R. Hardt and F.-H. Lin's regularity theory to the minimizer, \cite{HaLi-CPAM}.

In this paper we focus our attention on a non-compact, but topologically trivial, target manifold $N$ of non-positive curvature. Such a manifold $N$ is usually said to be Cartan-Hadamard. Under these assumptions, a solution to the Dirichlet problem has been given by Fuchs, \cite[Theorem 5.1]{Fu2}. See also the more recent \cite{FR-GeoDed}. In order to win the lack of compactness of the target, the proof given in \cite{Fu2} needed to deeply exploit Fuchs' regularity theory for constrained $p$-minimizers. Our main purpose, here, is to show that, even if $N$ is noncompact, one can prove a posteriori a uniform bound for the solution. In particular, thanks to a gluing\&compactification argument, the problem can be reduced to the closed one, at least in case the target is either a surface or rotationally symmetric.

Actually, we feel that the geometric construction introduced in this paper will be useful in more general settings where the analytic problem is  related to different functionals. Indeed, as it will be clear from the proof, the relevant properties of the $p$-energy required by the method we propose are: (a) the solvability of the problem when the target is compact and (b) a maximum principle for regular enough solutions. 

It's worthwhile to remark that one crucial point in the previous works \cite{Fu2,FR-GeoDed} is a quite implicit use of a tight relation between two different notions of bounded Sobolev maps: a first one, that we could call \textit{intrinsic}, is defined in a global coordinate chart of the target space. A second one, somewhat more standard and called \textit{extrinsic}, uses a proper isometric embedding of the target into a Euclidean space of sufficiently large dimension. In a future paper, \cite{PV-preparation}, we shall investigate carefully the relations between these two notions and we will point out some interesting consequences.

The starting point of the present investigation is that
the only interesting case involves target manifolds without compact quotients for, otherwise, the non-compact problem can be
reduced to the compact one where the  machinery alluded to above can be
applied without changes.

\begin{prp}
\label{prop_cocompact}Let $\left(  M,g\right)  $ be a compact, $m$-dimensional
Riemannian manifold with smooth boundary $\partial M\neq\emptyset$ and let
$\left(  N,h\right)  $ be a complete, Riemannian manifold of dimension $n$ such that its universal cover supports a strictly convex exhaustion function. Assume that there exists a subgroup $\Gamma$
of isometries of $N$ acting freely, properly and co-compactly on $N$. Then,
for any $p\geq2$ and for every $f\in C^{0}\left(M,N\right) \cap Lip(\partial M,N)$, the homotopy $p$-Dirichlet problem has a
solution $u\in C^{1\,,\alpha}\left(  \operatorname{int} (M),N\right)  \cap C^{0}\left(
M,N\right)  $. Moreover, the solution is unique provided $N$ has non-positive sectional curvature.
\end{prp}

We aim at facing the general situation where either we have no information on the structure of the isometry group of $N$ or it is known that $N$ has no
compact quotients. 

\begin{thm}\label{existence_2}
Let $\left(  M,g\right)  $ be a compact, $m$-dimensional Riemannian manifold
with smooth boundary $\partial M\neq\emptyset$ and let $N$ be an $n$-dimensional simply connected manifold of non-positive curvature $N^n_\sigma$ which is either rotationally symetric or $2$ dimensional. Then, for any $p\geq2$ and any given
$f\in C^{0}\left(  M,N\right)  \cap Lip\left(  \partial M,N\right)  $, the
$p$-Dirichlet problem%
\begin{equation}\label{p-pb}
\left\{
\begin{array}
[c]{ll}%
\Delta_{p}u=0 & \text{on }M\\
u=f & \text{on }\partial M,
\end{array}
\right.
\end{equation}
has a unique solution $u\in C^{1,\alpha}\left(  \operatorname{int} (M),N\right)  \cap C^{0}\left(  M\right)  $.
\end{thm}
\section{Scheme of the proofs}
We start this section giving the simple proof of Proposition \ref{prop_cocompact}.
\begin{proof}[Proof (of Theorem \ref {prop_cocompact})]
By assumption, $N^{\prime}=N/\Gamma$ is a compact, aspherical Riemannian manifold covered by $N$ via the quotient projection $P:N\rightarrow N^{\prime}$. The
original datum $f$ projects to a new function $P\left(  f\right)
:M\rightarrow N^{\prime}$ which, in turn, can be used to state the
corresponding $p$-Dirichlet problem%
\[
\left\{
\begin{array}
[c]{ll}%
\Delta_{p}u^{\prime}=0 & \text{on }M\\
u^{\prime}=P\left(  f\right)   & \text{on }\partial M.
\end{array}
\right.
\]

Thanks to the analysis of the compact target case provided in Part 1, this problem admits a solution $u^{\prime}\in C^{1\,,\alpha}\left(  \operatorname{int} (M),N^{\prime}\right)
\cap C^{0}\left(  M,N^{\prime}\right)  $ in the homotopy class of $P\left(
f\right)  $ relative to $\partial M$. Let $H^{\prime}:[0,1]\times M\rightarrow N^{\prime}$ be such a homotopy. The classical theory of
fibrations (see e.g. \cite{Hatcher}) then tells us that $H^{\prime}$ lifts to a homotopy $H:[0,1]\times M\rightarrow N$ satisfying $H\left( 1, x\right)  =f\left(  x\right)
$. The homotopy $H$ is relative to $\partial M$
because, for every $y\in\partial M,$ $H\left(  [0,1] \times \left\{  y\right\}
\right)$ is contained in the (discrete) fibre over
$P\left(  f\right)  \left(  y\right)  $. Let $u\left(  x\right)  =H\left(
0,x\right)  .$ Since $P$ is a local isometry and $P(u)=u^{\prime}$, then $u$ is  $p$-harmonic in $M$ \ of
class $C^{1\,,\alpha}\left(  \operatorname{int} (M),N\right)  \cap C^{0}\left(
M,N\right)  $. On the other hand, using the fact that $H$ is relative to
$\partial M$ we deduce that $u=f$ on $\partial M$. This proves that the
original homotopy $p$-Dirichlet problem has a solution. In case $\nsect\leq0$, uniqueness follows
easily from the following few facts: (a) solutions of the homotopy
$p$-Dirichlet problem with target $N$ projects to solutions of the
corresponding problem with target $N^{\prime}$; (b) in case of compact
targets, the solution is unique; (c) liftings are uniquely determined by their
values at a single point.
\end{proof}

The approach we propose to prove Theorem \ref{existence_2} inspires to the reduction procedure used to obtain Proposition
\ref{prop_cocompact}. This latter implies that the Dirichlet problem is easily solved when $N$ has a compact quotient, but this is not the case for a
general Cartan-Hadamard model manifold $N^{n}_{\sigma}$. The possible lack of discrete, co-compact isometry
subgroups is overcome by using a combination of cut\&paste and periodization arguments. Namely, we will show that it is possible to perturb the metric of $N^{n}_{\sigma}$ in the
exterior of a fixed geodesic ball in $N^{n}_{\sigma}$ such that the complete manifold thus obtained is again Cartan-Hadamard and has compact quotients. A new maximum principle
for the composition of the $p$-harmonic map and the convex distance function of $N^{n}_{\sigma}$ then gives that this perturbation does not affects the solution to the original problem.
The uniqueness part of the theorem can be clearly considered as a bypass product of the reduction to the compact case. We recall also that a
comprehensive uniqueness result for general complete targets with non-positive curvature was obtained in Part 1.

To perform the cut\&past procedure we need a local explicit control on the sectional curvatures of $N$. To this purpose, we first focus our attention on rotationally symmetric targets. That is, having fixed a smooth
function $\sigma:[0,+\infty)\rightarrow\lbrack0,+\infty)$ satisfying%
\begin{equation}
\sigma^{(2k)}\left(  0\right)  =0\text{, }\forall k \in \mathbb{N},\quad\sigma^{\prime}\left(  0\right)
=1,\quad\sigma\left(  r\right)  >0\text{, }\forall r>0,\label{model1}%
\end{equation}
we shall denote by $N_{\sigma}^{n}$ the smooth $n$-dimensional Riemannian
manifold given by%
\begin{equation}
\left(  \lbrack0,+\infty)\times\mathbb{S}^{n-1},dr^{2}+\sigma^{2}\left(
r\right)  d\theta^{2}\right)  ,\label{model2}%
\end{equation}
where $d\theta^{2}$ denotes the standard metric on $\mathbb{S}^{n-1}$.
Clearly, $N_{\sigma}^{n}$ is diffeomorphic to $\mathbb{R}^{n}$ and geodesically complete for any choice of $\sigma$. Usually, $N_{\sigma}^{n}$ is
called a model manifold with warping function $\sigma$ and pole $0$. The $r$-coordinate in the expression (\ref{model2}) of the metric represents the
distance from the pole. 
Standard formulas for
warped product metrics reveal that%
\begin{equation}\label{sec_rot}
\sect_{rad}=-\frac{\sigma^{\prime\prime}}{\sigma},\quad \sect_{tg}=\frac{1-\left(
\sigma^{\prime}\right)  ^{2}}{\sigma^{2}}%
\end{equation}
Thus, in particular, the model manifold $N_{\sigma}^{n}$ is Cartan-Hadamard if and only if%
\[
\sigma^{\prime\prime}\geq0.
\]
%

We point out that, when the Cartan-Hadamard target is $2$-dimensional, the first equation in \eqref{sec_rot} defines its Gaussian curvature in polar
coordinates regardless of any rotational symmetry condition. Namely, given a $2$-dimensional
Cartan-Hadamard manifold $(N,h_N)$, in the global geodesic chart $(r,\theta)$ around some fixed pole $o\in N$ the metric $h_N$ can be expressed as
\[
h_N|_{(r,\theta)}=dr^2 + \nu^2(r,\theta)d\theta^2.
\]
Direct computations show that the only
(radial) sectional curvature of $N^2_\nu$ satisfies at any point $(r,\theta)$ the formula
\begin{equation}\label{2-sec}
\sect(r,\theta)=\sect_{rad}(r,\theta)=-\nu^{-1}(r,\theta)\frac{\partial^2\nu(r,\theta)}{\partial r^2}.
\end{equation}

\section{Gluing model manifolds keeping $\sect\leq0$}\label{sec_gluing-curv}

In this Section we show that, in some sense, it is possible to prescribe a
hyperbolic infinity to a Cartan-Hadamard model, as well as to a generic Cartan-Hadamard $2$-manifold, without violating the non-positive
curvature condition.  

%
%
%
%
%
%

\begin{theorem}\label{th_gluing}
Let $N$ be a rotationally symmetric (resp. $2$ dimensional) Cartan-Hadamard manifold. Fix $\bar{R}>0$. Then, for every $R>\bar R$ there exist a
$k=k(R)>>1$ and a Cartan-Hadamard $M_{\tau}^{n}$ such that:

\begin{enumerate}
\item[(i)] $B_{\bar{R}}^{N}\left(  0\right)  \subset M_{\tau}^{n}$.

\item[(ii)] $M_{\tau}^{n}\backslash$ $B_{R}^{M}\left(  0\right)
=\hh_{k}^{n}\backslash$ $B_{R}^{\hh_k^n}\left(  0\right)  $.
\end{enumerate}
\end{theorem}
%
%

\begin{proof}
We prove the theorem in case $N=N_{\rho}^{n}$ is rotationally symmetric. Replacing \eqref{sec_rot} with \eqref{2-sec}, the two dimensional case can be handled in in a completely analogous way.

Thanks to \eqref{sec_rot}, it is enough to produce a warping function $\tau:[0,+\infty)\rightarrow
\lbrack0,+\infty)$ satisfying the following requirements:

\begin{enumerate}
\item[(a)] $\tau(r)=\rho(r)$ on $[0,\bar R).$

\item[(b)] $\tau(r)=\sigma_k:=k^{-1/2}\sinh\left(k^{1/2}r\right)$ on $(R,+\infty)$.

\item[(c)] $\tau^{\prime}\geq1$ and $\tau^{\prime\prime}\geq0$ on
$[0,+\infty)$.
\end{enumerate}

\noindent To this end, let $\bar{R}<R_{1}<R_{2}<R$. By the assumptions on
$\sigma$, we can choose $k=k\left(  R_{1},R_{2}\right)  >0$ large enough so
that%
\begin{equation}\label{weak_rel}
\rho^{\prime}\left(  R_{1}\right)  \leq\frac{\sigma_{k}\left(  R_{2}\right)
-\rho\left(  R_{1}\right)  }{R_{2}-R_{1}}\leq\sigma_{k}^{\prime}\left(
R_{2}\right)  .
\end{equation}
Define%
\[
\tau_{1}\left(  r\right)  =\left\{
\begin{array}
[c]{ll}%
\rho\left(  r\right)   & \text{on }[0,R_{1})\\
\rho\left(  R_{1}\right)  +\frac{\sigma_{k}\left(  R_{2}\right)  -\rho\left(  R_{1}\right)  }{R_{2}-R_{1}}r & \text{on }[R_{1},R_{2}]\\
\sigma_{k}\left(  r\right)   & \text{on }(R_{2},+\infty).
\end{array}
\right.
\]
Then, $\tau_{1}$ is a piecewise smooth, convex function with $\tau_{1}%
^{\prime}\geq1$. To complete the construction of $\tau$, it remains to
smoothing out the angles with a convex function. This can be done using the
approximation procedure described by M. Ghomi in \cite{Gh-PAMS}.
\end{proof}

%
%
%
%

\begin{remark}
{\rm As it is clear from the proof, Theorem \ref{th_gluing} holds for a class of ``external'' manifolds wider than hyperbolic spaces. In fact, all we need is relation \eqref{weak_rel} to hold.\\
}
\end{remark}

\section{Compact hyperbolic manifolds with large injectivity radii}\label{sec_inj}

It is intuitively clear that actions of small discrete groups on a complete
Riemannian manifold give rise to large fundamental domains. The intuition is
confirmed in the next simple result.

\begin{lemma}
\label{lemma_fundom}Let $\left(  N,h\right)  $ be a complete Riemannian
manifold. Suppose that there exists a filtration%
\[
\Gamma_{0}\vartriangleright\Gamma_{2}\vartriangleright\Gamma_{3}%
\vartriangleright\cdot\cdot\cdot\vartriangleright\Gamma_{k}\vartriangleright
\cdot\cdot\cdot\vartriangleright\left\{  1\right\}
\]
of discrete groups $\Gamma_{k}\subset\mathrm{Iso}\left(  N\right)  $ acting
freely and properly on $N$. Then, for every arbitrarily large ball
$B_{R}\left(  p\right)  $, there exists $K>0$ such that the following holds: for every $k>K$ we
find a fundamental domain $\Omega_{k}$ of $\Gamma_{k}$ containing $p$ and
satisfying%
\begin{equation}
B_{R}^{N}\left(  p\right)  \subset\subset\Omega_{k}. \label{tower2}%
\end{equation}

\end{lemma}

\begin{proof}
Let $D_{k}\left(  p\right)  $ be the Dirichlet domain of $\Gamma_{k}$ centered
at $p$. Recall that $D_{k}\left(  p\right)  =\cap_{\gamma\in\Gamma_{k}%
}H_{\gamma}\left(  p\right)  $ where%
\[
H_{\gamma}\left(  p\right)  =\left\{  x\in N:d_{N}\left(  x,p\right)
<d_{N}\left(  x,\gamma\cdot p\right)  \right\}  .
\]
One can easily verify that if $B_{R}^{N}\left(  p\right)  \cap(N\backslash
D_{k}\left(  p\right)  )\neq\emptyset$ then%
\begin{equation}
B_{R}^{N}\left(  p\right)  \cap\gamma\cdot B_{R}^{N}\left(  p\right)
\neq\emptyset, \label{tower3}%
\end{equation}
for some $\gamma\in\Gamma_{k}\subset\Gamma_{0}$. Since $\Gamma_{0}$ acts
properly on $N$ \ it follows that (\ref{tower3}) can be satisfied for at most
a finite number of $\gamma_{1},...,\gamma_{N}\in\Gamma_{0}$. To conclude the
validity of (\ref{tower2}), we now use that $\cap\Gamma_{k}=\left\{
1\right\}  $ and, therefore, $\gamma_{1},...,\gamma_{N}\notin\Gamma_{k},$
\ for every large enough $k.$
\end{proof}

A case of special interest is obtained by taking $N=\mathbb{H}_{-k^{2}}^{n},$
the standard hyperbolic spaceform of constant curvature $-k^{2}<0$. If
$\Gamma$ is a co-compact discrete group of isometries acting freely and
properly on $\mathbb{H}_{-k^{2}}^{n}$, the corresponding Riemannian orbit
space $\mathbb{H}_{-k^{2}}^{n}/\Gamma$ is named a compact hyperbolic manifold
(of constant curvature $-k^{2}$). 
The existence of a co-compact discrete group of
isometries of $\mathbb{H}_{-k^{2}}^{n}$ with large fundamental domain is equivalent to the existence of a compact
quotient manifold with large injectivity radius. 
%
The following result was first observed in
\cite{Fa-lectures}, see p.74.

\begin{proposition}
\label{prop_inj}Let $n\geq0$, $R>0$ and $p\in\mathbb{H}_{-k^{2}}^{n}$. Then,
there exists a co-compact, discrete group $\Gamma$ of isometries of
$\mathbb{H}_{-k^{2}}^{n}$ acting freely and properly on $\mathbb{H}_{-k^{2}%
}^{n}$ and whose fundamental domain $\Omega$ containing $p$ satisfies%
\[
\mathbb{B}_{R}\left(  p\right)  \subset\subset\Omega.
\]
Equivalently,%
\[
\operatorname{inj}\left(  \mathbb{H}_{-k^{2}}^{n}/\Gamma\right)  \geq R.
\]

\end{proposition}

\begin{proof}
By a result of A. Borel \cite{Borel-Topo}, $\mathbb{H}_{-k^{2}}^{n}$ has a co-compact, discrete
group of isometries $\Gamma_{0}$ acting freely and properly. According to a
result by A. Malcev, $\Gamma_{0}$ is residually finite, i.e., there exists a
filtration%
\[
\Gamma_{0}\vartriangleright\Gamma_{2}\vartriangleright\Gamma_{3}%
\vartriangleright\cdot\cdot\cdot\vartriangleright\Gamma_{k}\vartriangleright
\cdot\cdot\cdot\vartriangleright\left\{  1\right\}
\]
satisfying \thinspace$\lbrack\Gamma_{k}:\Gamma_{k-1}]=\left\vert \Gamma
_{k}/\Gamma_{k-1}\right\vert <+\infty$. To conclude, we now apply Lemma
\ref{lemma_fundom}.
\end{proof}

\section{A maximum principle for $p$-harmonic maps}\label{max}

It is well known, and an easy consequence of the composition law of the
Hessians, that by composing a harmonic map $u:M\rightarrow N$ with a convex
function $h:N\rightarrow\mathbb{R}$ gives a subharmonic function $v=h\circ
u:M\rightarrow\mathbb{R}$, i.e., $\Delta v\geq0$. In particular, if $M$ is
compact with smooth boundary $\partial M\neq0$ and $N$ is Cartan-Hadamard, we
can choose $h\left(  x\right)  =d_{N}^{2}\left(  x,o\right)  $ and apply the
usual maximum principle to conclude that the image $u\left(  M\right)  \subset
N$ is confined in a ball $B_{R}^{N}\left(  o\right)  $ of radius $\ R>0$
depending only on the values of $u$ on $\partial\Omega$, namely,
$R=\max_{\partial\Omega}d_{N}\left(  u,o\right)  $. It was
proved in \cite{Ve-Manuscripta} that, in general, the nice composition
property of harmonic maps does not extend to $p$-harmonic maps, $p>2$.
Nevertheless, we are able to recover the above conclusion \ thus establishing
a new maximum principle for the composition of a $p$-harmonic map and a convex function.

\begin{theorem}\label{th_mp}Let $M$ be a compact Riemannian manifold
with boundary $\partial M\neq\emptyset$, and let $u\in C^1(M,N)$ be a
$p\left( >1\right)  $-harmonic map. Assume that $N$ supports a smooth convex
function $f:N\rightarrow\mathbb{R}$. Set $w=f\circ u:M\rightarrow\mathbb{R}$.
Then%
\[
\sup_{M}w=\sup_{\partial M}w.
\]

\end{theorem}

\begin{proof}
We give the proof for $p\geq 2$ without taking care of the regularity issues. Similar distributional computations permit to deal with the general case. 

Let $w^{\ast}=\sup_{\partial M}w$ and, by contradiction, suppose that
$w\left(  x_{0}\right)  >w^{\ast}$ for some $x_0\in\operatorname{int} (M)$. Fix $0<\varepsilon<<1$ so that $w\left(
x_{0}\right)  -w^{\ast}>2\varepsilon$. \ Let $\lambda:\mathbb{R}%
\rightarrow\lbrack0,1]$ \ satisfy $\lambda^{\prime}\geq0$, $\lambda^{\prime
}>0$ on $\left(  \varepsilon,+\infty\right)  $, $\lambda=0$ on $(-\infty
,\varepsilon]$. Define the vector field%
\[
Z=\left\vert du\right\vert ^{p-2}\lambda(w-w^{\ast})\nabla w
\]
and note that \textrm{supp}$Z\subset\operatorname{int} (M)$. Direct computations show
that%
\begin{align*}
\operatorname{div}Z &  =\lambda^{\prime}\circ(w-w^{\ast})\left\vert du\right\vert
^{p-2}\left\vert \nabla w\right\vert ^{2}\\
&  +\lambda\circ(w-w^{\ast})\operatorname{tr}\operatorname{Hess}\left(  f\right)  \left(  \left\vert
du\right\vert ^{p-2}du,du\right)  \\
&  +\lambda\circ(w-w^{\ast})df\left(  \Delta_{p}u\right)  \\
&  \geq\left\vert \nabla w\right\vert ^{2}\left\vert du\right\vert ^{p-2}%
\lambda^{\prime}\circ(w-w^{\ast}),
\end{align*}
and applying the divergence theorem we get%
\[
0\leq\int_{M}\left\vert \nabla w\right\vert ^{2}\left\vert du\right\vert
^{p-2}\lambda^{\prime}\circ(w-w^{\ast})\leq\int_{M}\operatorname{div}Z=0.
\]
This proves that%
\begin{equation}
\left\vert \nabla w\right\vert ^{2}\left\vert du\right\vert ^{p-2}=0\text{ on
}M_{\varepsilon},\label{mp1}%
\end{equation}
where we have denoted with $M_\epsilon$ the connected component containing $x_0$ of the open set
\texttt{}\[
\left\{  x\in M:w-w^{\ast}-\varepsilon>0\right\}  .
\]
Since, by (\ref{mp1}), $dw=df\left(  du\right)  =0$ where $du\neq0$ and
$dw=df\left(  du\right)  =0$ \ where $du=0$, it follows that $w$ is constant on
$M_{\varepsilon}$ \ and this easily gives the desired contradiction.
\end{proof}

\section{Proof of the main results}\label{sec_main}

In this last Section we put all the previous ingredients together to get a proof of Theorems \ref{existence_2}.\medskip

The boundary datum $f$ has image confined in a ball $B_{R_{0}}^{N}\left(
0\right)  $ of $N^{n}$. Using Theorem \ref{th_gluing}, we glue
$B_{R_{0}}^{N}\left(  0\right)  $ to the exterior of a large ball in the
hyperbolic spaceform $\mathbb{H}_{-k^{2}}^{n}$ of sufficiently negative curvature
$-k^{2}<<-1$, say $\mathbb{H}_{-k^{2}}^{n}\backslash\mathbb{B}_{R_{1}}\left(
0\right)  $, $R_{1}>>R_{0}$, thus obtaining a new Cartan-Hadamard rotationnally symmetric (resp. $2$ dimenstional) manifold $\left(
N^{\prime},h^{\prime}\right)  $. On the
other hand, by Proposition \ref{prop_inj}, $\mathbb{H}_{-k^{2}}^{n}$ has
compact quotients with arbitrarily large injectivity radii. Accordingly, we
can choose a discrete subgroup $\Gamma$ of isometries acting freely and
co-compactly on $\mathbb{H}_{-k^{2}}^{n}$ in such a way that $\mathbb{B}%
_{R_{1}}\left(  0\right)  $ is contained in a relatively compact, fundamental
domain of the action, say $\mathbb{B}_{R_{1}}\left(  0\right)  \subset
\subset\Omega$. Making use of $\Gamma$ we extend the deformed metric of $\bar{\Omega}$ periodically thus obtaining a new Riemannian manifold $N^{\prime \prime}$ diffeomorphic to $\mathbb{H}^m_{-k^2}$. More precisely, the metric $h^{\prime \prime}$ of $N^{\prime \prime}$ is defined by setting
\[
h_{\gamma\cdot p}^{\prime\prime}=\left(  \gamma^{-1}\right)  _{\gamma\cdot p}^{\ast}%
h_{p}^{\prime}.
\]
Since $h^{\prime}$ is hyperbolic in a neighborhood of $\partial\Omega$, the
definition of $h^{\prime\prime}$ is well posed. Moreover, $\left(
N^{\prime\prime},h^{\prime\prime}\right)  $ has non-positive curvature, hence it is Cartan-Hadamard, and, by
construction, $\Gamma$ acts freely and co-compactly by isometries on $N''$. In
particular, each copy of $\Omega$ contains an isometric image of $B_{R_{0}%
}^{N}\left(  0\right)  $. Now, we take the quotient manifold $N^{\prime\prime
}/\Gamma$ which is compact and covered by
$N^{\prime\prime}$ via the quotient projection $P:N^{\prime\prime}\rightarrow
N^{\prime\prime}/\Gamma$. By construction, the original datum $f$ well defines $f^{\prime\prime
}=f:M\rightarrow N^{\prime\prime}$. Applying Proposition \ref{prop_cocompact} we get a unique solution $u^{\prime \prime} \in C^0(M,N^{\prime \prime}) \cap C^{1,\alpha}(\operatorname{int} (M),N^{\prime \prime})$ to the Dirichlet problem
\[
\left\{
\begin{array}
[c]{ll}%
\Delta_{p}u^{\prime\prime}=0 & \text{on }M\\
u^{\prime\prime}=f^{\prime\prime} & \text{on }\partial M.
\end{array}
\right.
\]
To complete the argument, it remains to show that, actually, $u^{\prime\prime
}$ gives rise to a solution of the original problem. This clearly follows if
we are able to show that its image is confined in $B_{R_{0}}^{N}\left(
0\right)  \subset N^{\prime\prime}$. To prove that this is the case, we recall that $N^{\prime\prime}$ is Cartan-Hadamard and, therefore, the
function $d_{N^{\prime\prime}}^{2}\left(  y,0\right)  $ is smooth and strictly convex. By means of Theorem
\ref{th_mp}, we deduce that $d_{N^{\prime\prime}}^{2}\left(  u^{\prime\prime
},0\right)  $ achieves its maximum on $\partial M$. To conclude, it suffices
to recall that $f\left(  M\right)  \subset B_{R_{0}}^{N}\left(  0\right)  $
and to use the equality $u^{\prime\prime}=f$ on $\partial M$.

%
\acknowledgement\nonumber
We are indebted to Fran\c{c}ois Fillastre for some conversations concerning closed hyperbolic manifolds which have revealed very useful to the draft of Subsection \ref{sec_inj}.\\
The second author was partially supported by the \textit{INdAM Fellowships in Mathematics and/or Applications for Experienced Researchers cofunded by Marie
Curie} and by the \textit{Gruppo Nazionale per l'Analisi Matematica, la Probabilit\`a e le loro Applicazioni (GNAMPA)}.

\end{document}